\documentclass[a4paper,11pt]{article}

\usepackage{amsmath,amsthm,amsfonts,amssymb,graphics,epsfig,color}

\usepackage[left=2cm,top=3cm,right=2cm,bottom=3cm,bindingoffset=0.5cm]{geometry}
\usepackage[english]{babel}
\usepackage{verbatim}
\usepackage{amscd,enumerate}
\usepackage{epstopdf}
\usepackage{graphicx}
\usepackage{multicol}

\usepackage{mathrsfs}

\newtheorem{teo}{Theorem}[section]

\newtheorem{lema}{Lemma}[section]

\newtheorem{obs}{Remark}[section]

\newcommand{\Ste}{\text{Ste}}
\newcommand{\Bi}{\text{Bi}}

\newcommand\be{\begin{equation}}
\newcommand\ee{\end{equation}}
\DeclareMathOperator\erf{erf}

\newcommand\pder[2][]{\ensuremath{\frac{\partial#1}{\partial#2}}}

\newcommand{\Nr}{{\overline{N}}}
\newcommand{\Lr}{{\overline{L}}}

\newcommand{\vr}{{\overline{v}}}

\newcommand{\bint}{\displaystyle\int}
\newcommand{\dd}{{\rm{d}}}
\newcommand{\Pe}{{\rm{Pe}}}

\newcommand{\wtL}{\widetilde{L}}
\newcommand{\wtN}{\widetilde{N}}
\newcommand{\wtmu}{\widetilde{\mu}}
\newcommand{\wtk}{\widetilde{k}}

\begin{document}

\title{Stefan problems for the  diffusion-convection equation with temperature-dependent thermal coefficients}
\author{
Julieta Bollati$^{1}$, Adriana C. Briozzo $^{1}$ \\
\small {{$^1$} CONICET-Depto. Matem\'atica, FCE, Univ. Austral, Paraguay 1950} \\
\small {S2000FZF Rosario, Argentina.}\\
\small{Email: jbollati@austral.edu.ar, abriozzo@austral.edu.ar,}}

\date{}

\maketitle
\abstract{
Different one-phase Stefan problems for a semi-infinite slab are considered, involving a moving phase change material as well as  temperature
dependent thermal coefficients.
Existence of at least one similarity solution is proved imposing a Dirichlet, Neumann, Robin or radiative-convective boundary condition at the fixed face. The velocity that arises in the convective term of the  diffusion-convection equation is assumed to depend on temperature and time.
In each case, an equivalent ordinary differential problem is obtained giving rise to a system of an integral equation coupled with a condition for the parameter that characterizes the free boundary, which is solved though a double-fixed point analysis. Some solutions for particular thermal coefficients  are provided.

}


\smallskip


\section{Introduction}

Stefan problems constitute a broad field of study since they arise in different areas of engineering, biology, geoscience and industry \cite{AlSo1993,Ca1984,Cr1984,Lu1991}.
The classical one-phase Stefan problem models a phase-change thermal process that aims to describe the temperature of the material as well as the location of the interface that separate both phases. Mathematically it consists on finding a solution to the heat-conduction equation
in an unknown region which has also to be determined,
imposing an initial condition, boundary conditions, and the
Stefan condition at the moving interface.

In many physical processes, the phase change material is
allowed to move when the phase change occurs. 
Recently, in \cite{Tu2018} a
Stefan problem  which models the undergoing phase transition of a moving material where the 
phase change and heat distribution in the medium are affected from
both the conduction and convection of heat  was considered.

The diffusion-convection equation has multiple applications, for example, to ground water hydrology, oil reservoir engineering and  drug propagation in the arterial tissues.
More articles where a convective term is involved in the parabolic equation are \cite{BrNaTa1997,BrNaTa2010,KuRa2020C,Ro1982,SiKuRa2019A}.

  In particular, in  \cite{SiKuRa2019A} a Stefan problem with variable thermal coefficients and moving phase
change material was studied. It was considered a thermal conductivity and a specific heat whose  dependence on the temperature   was assumed to be linear.

Motivated by \cite{SiKuRa2019A,Tu2018}, in this paper we consider a free boundary problem in a semi-infinite domain $x>0$ for the nonlinear diffusion-convection equation with a convective term that involves temperature-dependent thermal coefficients.

 The  problem consists in finding the temperature $T=T(x,t)$ in the liquid region and the free boundary $x=s(t)$ such that:
\begin{subequations}
\begin{align}
& \rho(T) c(T) \frac{\partial T}{\partial t}=\frac{\partial}{\partial x} \left(k(T)\frac{\partial T}{\partial x}\right)-v(T)\pder[T]{x},& 0<x<s(t), \quad t>0, \label{Ec:Calor}\\
& T(0,t)=T^*, &t>0, \label{Cond:Temp x=0}\\
&  T(s(t),t)=T_m, &t>0, \label{Cond:TempCambioFase}\\
&  k\left(T(s(t),t)\right)\frac{\partial T}{\partial x}(s(t),t)=-\rho_0 \ell \dot s(t), &t>0, \label{Cond:Stefan}\\
& s(0)=0,\label{Cond:FronteraInicial}
\end{align}
\end{subequations}
where $\rho(T)$, $c(T)$ and $k(T)$ are the mass density, the specific heat and the thermal conductivity of the body, respectively 
defined as\begin{equation}
\begin{array}{lll}
\rho(T)(x,t)=\rho(T(x,t)),   \quad c(T)(x,t)=c(T(x,t)), \quad k(T)(x,t)=k(T(x,t)).\quad 
\end{array}
\end{equation}

Some  models involving temperature-dependent
thermal coefficients can  be found in \cite{BoNaSeTa2019-ChHr,BrNa2014-1,BrNa2014-2,CeSaTa2018-a,ChSu1974,NaTa2000,OlSu1987}.

 The presence of a convection term in equation (\ref{Ec:Calor}) represents the fact that the phase change material is allowed to move with an  unidirectionally speed $v=v(T)$.
Throughout this paper the unidirectional speed $v$ is given by 
\begin{equation}
v(T)=\dfrac{\mu(T)}{\sqrt{t}}\label{velocidad-convectiva}
\end{equation}
with $\mu(T)(x,t)=\mu(T(x,t))$. 

 \medskip

We assume that  $T_m$ is the phase change temperature and $T^*>T_m$ is the temperature imposed at the fixed face $x=0$. Condition \eqref{Cond:Stefan} represents the Stefan condition where $\rho_0>0$ is a constant mass density and $\ell>0$ is the latent heat  of fusion per unit mass.

We will  consider three more problems that arise replacing the Dirichlet condition at the fixed face by other type of conditions.

On one hand we stablish the problem with a Neumann condition at the fixed face, replacing (\ref{Cond:Temp x=0}) by
\begin{equation}\label{Cond:Flujo x=0}
k(T(0,t)) \pder[T]{x}(0,t)=-\dfrac{q}{\sqrt{t}},\qquad t>0, \tag{\ref{Cond:Temp x=0}$^\star$}
\end{equation} 
where  $q>0$ is a given constant and $-\tfrac{q}{\sqrt{t}}$ represents the prescribed flux at $x=0$. Some bibliography imposing this kind of condition can be found in \cite{BoNaSeTa2020,BrNaTa2007,NaTa2006,LoTa2001,Ta1981-1982}

On the other hand, we consider a problem governed by \eqref{Ec:Calor} and \eqref{Cond:TempCambioFase}-\eqref{Cond:FronteraInicial} where a Robin condition is imposed:
\begin{equation}\label{Cond:Convectiva x=0}
k(T(0,t))\pder[T]{x}(0,t)=\frac{h}{\sqrt{t}}\left[T(0,t)-T^* \right],\qquad t>0, \tag{\ref{Cond:Temp x=0}$^\dag$}
\end{equation} 
being $h$ the coefficient that characterizes the heat transfer at the fixed face and $T^*$ the bulk temperature applied at a neighbourhood of $x=0$ with $T^*>T(0,t)>T_m$  \cite{BrNa2016,BrNa2019,BrNa2015,BrTa2010-2,Ta2017}.

Finally we define the problem that arises  replacing the Dirichlet condition (\ref{Cond:Temp x=0}) by a radiative and convective condition 
\begin{equation}\label{Cond:RadiactivaConvectiva x=0}
k(T(0,t))\pder[T]{x}(0,t)=\frac{h}{\sqrt{t}}\left[T(0,t)-T^* \right]+\frac{\sigma \epsilon}{\sqrt{t}}\left[T^4(0,t)-{T^*}^4 \right],\qquad t>0, \tag{\ref{Cond:Temp x=0}$^{\dag\dag}$}
\end{equation} 
where $\sigma$ is the Stefan-Boltzmann constant and $\epsilon>0$ is the coefficient that characterizes the radiation shape factor, assuming \mbox{$T^*>T(0,t)>T_m$}. Notice that the first term of the r.h.s of condition \eqref{Cond:RadiactivaConvectiva x=0} coincides to the r.h.s of the Robin condition \eqref{Cond:Convectiva x=0}. This kind of boundary condition also appears in \cite{ChYe1975,HuLi2020,WaWaZeChYa2018,YaHu1979}

\medskip

The aim of this paper is to provide sufficient conditions on data in order to guarantee existence of at least one solution of a similarity type to four different  problems that differ from each other in the boundary condition imposed at the fixed face: temperature, flux, convective or radiative-convective condition.

  The manuscript is organised as follows. In Section 2 we analyse  the existence of at least one similarity solution to the problem governed by \eqref{Ec:Calor}-\eqref{Cond:FronteraInicial} where a constant temperature is imposed at the fixed face. Introducing the similarity variable,  an equivalent  ordinary differential problem is obtained, giving rise to nonlinear integral equation  coupled with a condition for the parameter that characterizes the free boundary. This system is solved by  a double fixed point analysis. 
{\color{black}In a similar way, in Section 3 we stablish  the existence of solution to the Stefan problem that arises when we replace the Dirichlet condition \eqref{Cond:Temp x=0} by  a Neumann one \eqref{Cond:Flujo x=0}. Section 4 is devoted to the study the problem where a Robin condition \eqref{Cond:Convectiva x=0} is imposed at $x=0$. Moreover, we analyse the convergence of the solution when $h\to +\infty$, and recover for $v=0$ a result given in \cite{BrNa2019}. Finally, in Section 5 we generalize Section 4 by showing that there exists at least one solution to the problem that arises when  considering a radiative-convective condition \eqref{Cond:RadiactivaConvectiva x=0} at the fixed face.

In the last section we present different solutions  obtained for some particular cases. On one hand, we consider  constant thermal coefficients and a velocity given by $v(T)=\frac{\mu(T)}{\sqrt{t}}$ with
\be 
\mu(T)=\rho_0 c_0 \sqrt{\alpha_0} \; \Pe,
\ee
where Pe denotes the Peclet number. The solutions given by \cite{Tu2018} for Dirichlet and Neumann condition at fixed face are recovered. 
On the other hand, we analyse the particular case when the thermal coefficients involved are linear functions of the temperature as  in  \cite{SiKuRa2019A}}.

\section{Dirichlet condition}

The following section is devoted to the analysis of the Stefan problem given by \eqref{Ec:Calor}-\eqref{Cond:FronteraInicial}.
 
{\color{black}Mathematically we will consider that   $\rho,c, k, v$ and $\mu$ are functions defined on $\mathscr{C}=C^0(\mathbb{R}_0^+\times\mathbb{R}^+)\cap L^{\infty}(\mathbb{R}_0^+\times\mathbb{R}^+)$
and we denote the norm of $\rho(T)$ by
\be
||\rho(T)||=\max\limits_{(x,t)\in\mathbb{R}_0^+\times\mathbb{R}^+} |\rho(T)(x,t)| .
\ee 
The same norm will be used for $c,k,v$ and $\mu$.

}

We will assume:
\begin{eqnarray}
&&\left\{ \begin{array}{ll}
{\rm (a)\ There\ exists\  } k_m>0 {\rm \ and\ } k_M>0 {\rm \  such\ that\ }\\[2mm]
\qquad k_m\leq k(T)\leq k_M,\qquad \forall\; T\in  \mathscr{C}.\\[2mm]
{\rm (b)\ There\ exists\  } \wtk>0 {\rm \  such\ that\ }\\[2mm]
\qquad ||k(T_1)-k(T_2)) ||\leq \wtk\; \|T_1-T_2\|,\;\; 
 \qquad \forall\;  T_1,T_2\in \mathscr{C}.
\end{array}\right. \label{Hipotesis-k}\qquad
\end{eqnarray}
\begin{eqnarray}
&&\left\{ \begin{array}{ll} 
{\rm (a)\ There\ exists\  } \gamma_m>0 {\rm \ and\ } \gamma_M>0 {\rm \  such\ that\ }\\[2mm]
\qquad \gamma_m\leq \rho(T) c(T)\leq \gamma_M,\qquad \forall\; T\in \mathscr{C}. \;\;\\[2mm]
{\rm (b)\ There\ exists\  } \widetilde{\gamma}>0 {\rm \  such\ that\ }\\[2mm]
\qquad ||\rho(T_1)c(T_1)-\rho(T_2)c(T_2)||\leq \widetilde{\gamma} \|T_1-T_2\|, \quad  \forall\; T_1,T_2\in \mathscr{C}. 
\end{array}\right. \label{Hipotesis-rhoc}
\end{eqnarray}
\begin{eqnarray}
&&\left\{ \begin{array}{ll} 
{\rm (a)\ There\ exists\  } \nu_m>0 {\rm \ and\ } \nu_M>0 {\rm \  such\ that\ }\\[2mm]
\qquad  \nu_m  \leq \mu(T)\leq \nu_M ,\qquad \forall \; T\in \mathscr{C}. \\[2mm]
{\rm (b)\ There\ exists\  } \widetilde{\nu}>0 {\rm \  such\ that\ }\\[2mm]
\qquad ||\mu(T_1)-\mu(T_2) ||\leq \widetilde{\nu}\;\|T_1-T_2\|,\quad \forall \; T_1,T_2\in  \mathscr{C}. \\[0mm] 
\end{array}\right. \qquad\qquad \label{Hipotesis-v}
\end{eqnarray}
\medskip
If we  introduce the following change of variables:
\begin{equation}\label{cambio-theta}
\theta=\dfrac{T-T^*}{T_m-T^*}>0,
\end{equation}
 we have
$$\pder[T]{t}=(T_m-T^*) \pder[\theta]{t},\qquad\qquad \pder[T]{x}=(T_m-T^*) \pder[\theta]{x},\qquad\qquad \pder[^2 T]{x^2}=(T_m-T^*) \pder[^2\theta]{x^2}.$$
Taking into account that  $T=T(\theta)=(T_m-T^*)\theta+T^*$ we can define the following functions:
\begin{align}
&\Lr(\theta)=L(T(\theta)),\qquad\qquad \Nr(\theta)=N(T(\theta)), \qquad \qquad
\vr(\theta)=\dfrac{v(T(\theta))}{\rho_0 c_0 }\qquad \label{LNv-Raya}
\end{align}
where 
\begin{align}
L(T)=\dfrac{k(T)}{k_0},\qquad\qquad\qquad  N(T)=\dfrac{\rho(T)c(T)}{\rho_0 c_0}\label{Def-LN}
\end{align}
and $k_0$, $\rho_0$, $c_0$ and $\alpha_0=\dfrac{k_0}{\rho_0  c_0}$ are the reference thermal conductivity, mass density, specific heat and thermal diffusivity, respectively.

Therefore,  the problem (\ref{Ec:Calor})-(\ref{Cond:FronteraInicial}) becomes:
\begin{subequations}
\begin{align}
& \Nr(\theta) \pder[\theta]{t}=\alpha_0 \frac{\partial}{\partial x}\left( \Lr(\theta)\pder[\theta]{x} \right)-\vr(\theta) \pder[\theta]{x}
,& 0<x<s(t), \quad t>0, \label{Ec:Calor-Theta}\\
& \theta(0,t)=0, &t>0, \label{Cond:Temp x=0-Theta}\\
&  \theta(s(t),t)=1, &t>0, \label{Cond:TempCambioFase-Theta}\\
&\Lr(\theta(s(t),t))\pder[\theta]{x}(s(t),t)=\dfrac{ \dot{s}(t)}{\alpha_0 \Ste} , &t>0, \label{Cond:Stefan-Theta}\\
& s(0)=0,\label{Cond:FronteraInicial-Theta}
\end{align}
\end{subequations}
where $
\Ste=\dfrac{(T^*-T_m)c_0}{\ell}>0$  is the Stefan number.

If we introduce the similarity variable $\xi=\dfrac{x}{2\sqrt{\alpha_0 t}}$ and assume a  similarity type solution
\be\label{similarity-variable}
\theta(x,t)=f(\xi),
\ee
then conditions (\ref{Cond:TempCambioFase-Theta})-(\ref{Cond:Stefan-Theta}) yields to a free boundary given by
\be 
s(t)=2\lambda\sqrt{\alpha_0 t}.
\ee
where $\lambda>0$ is a constant to be determined.

Let us define
\be\label{LNv-Estrella}
L^*(f)=\Lr(\theta),\qquad\qquad  N^*(f)=\Nr(\theta), \qquad\qquad  v^*(f)=\vr(\theta).
\ee
Then, \eqref{Ec:Calor} turns into the following ordinary differential equation
\be
\Big(L^*(f)f'(\xi)\Big)'+2 N^*(f)\; \xi\;f'(\xi)-\dfrac{2\sqrt{t}}{\sqrt{\alpha_0}} v^*(f)\; f'(\xi)=0 ,\qquad  0<\xi<\lambda. 
\ee
Taking into account that the unidirectional speed $v$ is given by \eqref{velocidad-convectiva}, and  from \eqref{LNv-Raya} and  \eqref{LNv-Estrella}, we define
\be\label{Def:muEstrella}
\mu^*(f)=\dfrac{v^*(f) \sqrt{t}}{\sqrt{\alpha_0}}=\dfrac{\mu(T)}{\sqrt{\rho_0 c_0 k_0}}.
\ee
Therefore we reduce the problem (\ref{Ec:Calor-Theta})-(\ref{Cond:FronteraInicial-Theta}) into  an ordinary differential problem defined by:
\begin{subequations}
\begin{align}
&\Big(L^*(f)f'(\xi)\Big)'+2 f'(\xi) \Big( N^*(f) \xi-\mu^*(f)\Big)=0 ,& 0<\xi<\lambda, \label{EcDif-Temp}\\
& f(0)=0,  \label{Condf=0-Temp}\\
&  f(\lambda)=1,  \label{Condf=1-Temp}\\
&L^*(f(\lambda)) f'(\lambda)=\dfrac{2\lambda}{\Ste}. \label{EcLambda-Temp}
\end{align}
\end{subequations}
From this ordinary differential problem \eqref{EcDif-Temp}-\eqref{EcLambda-Temp} we get that  $f$ and $\lambda$ must satisfy the following  integral equation
\be 
f(\xi)=\dfrac{\Phi(f)(\xi)}{\Phi(f)(\lambda)},\qquad 0\leq\xi\leq \lambda, \label{Volterra-Temp}
\ee
together with the following condition
\be
\dfrac{\Ste}{2} \dfrac{E(f)(\lambda)}{\Phi(f)(\lambda)}=\lambda,\label{CondLambda-Temp} 
\ee
where
\be
\Phi(f)(\xi)=\bint_0^{\xi} \dfrac{E(f)(z)}{L^*(f)(z)}\; \dd z,\label{Def-Phi}
\ee
with
\be
E(f)(z)=\dfrac{U(f)(z)}{I(f)(z)},\label{Def-E}
\ee
\be\label{Def-UI}
U(f)(z)=\exp\Big(2\bint_0^z \dfrac{\mu^*(f)(\sigma)}{L^*(f)(\sigma)}\; \dd\sigma\Big),\quad I(f)(z)=\exp\Big(2\bint_0^z \dfrac{\sigma N^*(f)(\sigma)}{L^*(f)(\sigma)}\; \dd\sigma\Big). 
\ee

\medskip
In order to analyse the existence of solution to the problem \eqref{Volterra-Temp}-\eqref{CondLambda-Temp}, let us study in the  first instance, the integral equation \eqref{Volterra-Temp}  assuming that  $\lambda>0$ is a fixed given constant.

Consider the space $C^0[0,\lambda]$ of continuous real-valued functions defined on $[0,\lambda]$  endowed with the supremum norm
\be
\Vert f\Vert=\max\limits_{\xi \in [0,\lambda]}\vert f(\xi)\vert.
\ee
Let us define the operator  $\mathcal{H}$ on $C^0[0,\lambda]$ given by the r.h.s of equation \eqref{Volterra-Temp}:
\be  
\mathcal{H}(f)(\xi)=\dfrac{\Phi(f)(\xi)}{\Phi(f)(\lambda)}.\label{Def-H}
\ee

{


Then, as $(C^0[0,\lambda], \Vert\cdot\Vert)$ is a Banach space we  use the fixed point Banach theorem to prove that for each $\lambda>0$ there exists a unique $f$ such that
\be
\mathcal{H}(f)(\xi)=f(\xi) ,\qquad 0\leq \xi\leq  \lambda\label{PtoFijo-Temp}
\ee
which is  the solution  to \eqref{Volterra-Temp}.

From the assumptions \eqref{Hipotesis-k}, \eqref{Hipotesis-rhoc} and \eqref{Hipotesis-v} we obtain that $L^*$, $N^*$ and $\mu^*$ are bounded and Lipschitz continuous. That is to say
\begin{eqnarray}
&&\left\{ \begin{array}{ll} L^*(f) \mbox{\ is\ such that:}\\[2mm]
{\rm (a)\ There\ exists\  } L_m=\tfrac{k_m}{k_0}>0 {\rm \ and\ } L_M=\tfrac{k_M}{k_0}>0 {\rm \  such\ that\ }\\[2mm]
\qquad L_m\leq L^*(f)\leq L_M,\qquad \forall f\in C^0(\mathbb{R}_0^+)\cap L^{\infty}(\mathbb{R}_0^+). \\[2mm]
{\rm (b)\ There\ exists\  } \wtL=\tfrac{\wtk(T^*-T_m)}{k_0}>0 {\rm \  such\ that\ }\\[2mm]
\qquad \|L^*(f_1)-L^*(f_2) \|\leq \wtL \|f_1-f_2\|,\quad \forall f_1,f_2\in C^0(\mathbb{R}_0^+)\cap L^{\infty}(\mathbb{R}_0^+). 
\end{array}\right. \label{Hipotesis-L*}
\end{eqnarray}
\begin{eqnarray}
&&\left\{ \begin{array}{ll} N^*(f) \mbox{\ is\ such that:}\\[2mm]
{\rm (a)\ There\ exists\  } N_m=\tfrac{\gamma_m}{\rho_0 c_0}>0 {\rm \ and\ } N_M=\tfrac{\gamma_M}{\rho_0 c_0}>0 {\rm \  such\ that\ }\\[2mm]
\qquad N_m\leq N^*(f)\leq N_M,\quad \forall f\in C^0(\mathbb{R}_0^+)\cap L^{\infty}(\mathbb{R}_0^+). \\[2mm]
{\rm (b)\ There\ exists\  } \wtN=\tfrac{\widetilde{\gamma} (T^*-T_m)}{\rho_0 c_0}>0 {\rm \  such\ that\ }\\[2mm]
\qquad \|N^*(f_1)-N^*(f_2) \|\leq \wtN \|f_1-f_2\|,\; \forall f_1,f_2\in C^0(\mathbb{R}_0^+)\cap L^{\infty}(\mathbb{R}_0^+).
\end{array}\right. \label{Hipotesis-N*}
\end{eqnarray}
\begin{eqnarray}
&&\left\{ \begin{array}{ll} \mu^*(f) \mbox{\ is\ such that:}\\[2mm]
{\rm (a)\ There\ exists\  } \mu_m=\tfrac{\nu_m}{\sqrt{\rho_0 c_0 k_0}}>0 {\rm \ and\ } \mu_M=\tfrac{\nu_M}{\sqrt{\rho_0 c_0 k_0}}>0 {\rm \  such\ that\ }\\[2mm]
\qquad \mu_m\leq \mu^*(f)\leq \mu_M,\qquad \forall f\in C^0(\mathbb{R}_0^+)\cap L^{\infty}(\mathbb{R}_0^+). \\[2mm]
{\rm (b)\ There\ exists\  } \wtmu=\tfrac{\widetilde{\nu}(T^*-T_m)}{\sqrt{\rho_0 c_0 k_0}}>0 {\rm \  such\ that\ }\\[2mm]
\qquad \|\mu^*(f_1)-\mu^*(f_2) \|\leq \wtmu\; \|f_1-f_2\|,\quad \forall f_1,f_2\in C^0(\mathbb{R}_0^+)\cap L^{\infty}(\mathbb{R}_0^+).\\[0mm] 
\end{array}\right. \label{Hipotesis-mu*}
\end{eqnarray}

Let us present now some preliminary results that will allow us to prove the existence and uniqueness of solution to equation \eqref{PtoFijo-Temp}.

\begin{lema}\label{LemaCotasUIEPhi} For all $z\in [0,\lambda]$ the following bounds hold
\begin{align}
& \exp\left(2\frac{\mu_M}{L_M}z\right)\leq  U(f)(z)\leq \exp\left(2\dfrac{\mu_M}{L_m}z \right),\label{cota-U}\\[0.15cm]
& \exp\left( \dfrac{N_m}{L_M}\right)\leq   I(f)(z)\leq \exp\left( \dfrac{N_M}{L_m}z^2\right),\label{cota-I}\\[0.15cm]
&\exp\left(-\tfrac{N_M}{L_m}z^2 \right)\leq \tfrac{\exp\left(2\tfrac{\mu_m}{L_M}z\right)}{\exp\left(\tfrac{N_M}{L_m}z^2 \right)}\leq E(f)(z)\leq \tfrac{\exp\left(2\tfrac{\mu_M}{L_m}z\right)}{\exp\left(\tfrac{N_m}{L_M}z^2\right)}\leq  \exp\left(2\tfrac{\mu_M}{L_m}z \right),\label{cota-E} \\[0.15cm]
&\tfrac{z}{L_M}\exp\left(-\tfrac{N_M}{L_m}z^2 \right)\leq \tfrac{\sqrt{\pi}}{2} \tfrac{\sqrt{L_m}}{L_M \sqrt{N_M}} \erf\left(\sqrt{\tfrac{N_M}{L_m}}z \right)\leq \Phi(f)(z)\leq \tfrac{1}{2\mu_M} \exp\left(2\tfrac{\mu_M}{L_m}z \right).\label{cota-Phi}
\end{align}
\end{lema}
\begin{proof}
The proof follows immediately from the definitions of $U,I,E,\Phi$ using assumptions \eqref{Hipotesis-L*}-\eqref{Hipotesis-mu*}.
\end{proof}
\medskip

\begin{lema} \label{LemaLipschitzUIEPhi}Given $\lambda>0$, for all $z \in[0,\lambda]$ and  $f_1,f_2\in C^0[0,\lambda]$  the following inequalities hold
\begin{align}
|U(f_1)(z)-U(f_2)(z)|&\leq  D_1(z) \|f_1-f_2 \|,\label{Lipschitz-U}\\[0.15cm]
|I(f_1)(z)-I(f_2)(z)|&\leq  D_2(z) \|f_1-f_2 \|,\label{Lipschitz-I}\\[0.15cm]
|E(f_1)(z)-E(f_2)(z)|&\leq  D_3(z) \|f_1-f_2 \|,\label{Lipschitz-E}\\[0.15cm]
|\Phi(f_1)(z)-\Phi(f_2)(z)|&\leq  \lambda\; D_4(\lambda) \|f_1-f_2 \|,\label{Lipschitz-Phi}
\end{align}
where
\begin{eqnarray}\label{D1-4}
\begin{array}{lll}
&D_1(z)=\dfrac{2\exp\left(\tfrac{2\mu_M}{L_m} \right)}{L_m^2}z \left( \mu_M \wtL+L_m \wtmu \right), \qquad \\[0.35cm]
&  D_2(z)=\dfrac{\exp\left(\tfrac{N_M}{L_m}z^2 \right)}{L_m^2}z^2 \left( N_M \wtL+L_m \wtN \right),  \\[0.35cm]
&D_3(z)= \exp\left( \dfrac{N_M}{L_m}z^2\right) D_1(z)+ \exp\left(2\dfrac{\mu_M}{L_m}z \right) D_2(z), \\[0.35cm]
&D_4(\lambda)= \dfrac{1}{L_m^2} \Big(\wtL \exp\left(2\lambda \tfrac{\mu_M}{L_m} \right) + L_m D_3(\lambda)\Big).
\end{array}
\end{eqnarray}
\end{lema}

\begin{proof}
Applying the mean value theorem and taking into account  assumptions \eqref{Hipotesis-L*}-\eqref{Hipotesis-mu*} we obtain that
\begin{align*}
&|U(f_1)(z)-U(f_2)(z)|\leq 2\exp\left(2\frac{\mu_M}{L_m}z \right) \bint_0^z \left|\dfrac{\mu^*(f_1)(\sigma)}{L^*(f_1)(\sigma)}-\dfrac{\mu^*(f_2)(\sigma)}{L^*(f_2)(\sigma)} \right| \dd \sigma \\[0.15cm]
&\leq 2\exp\left(2\frac{\mu_M}{L_m}z \right) \left\lbrace \bint_0^z  \left|  \tfrac{\mu^*(f_1)(\sigma)}{L^*(f_1)(\sigma) L^*(f_2)(\sigma)}\right|  \left| L^*(f_2)(\sigma)-L^*(f_1)(\sigma)\right| \dd \sigma \right.\\[0.15cm]
&+\left.\bint_0^z  \left|  \tfrac{1}{L^*(f_2)(\sigma)}\right|  \left| \mu^*(f_2)(\sigma)-\mu^*(f_1)(\sigma)\right| \dd \sigma\right\rbrace\\[0.15cm]
&\leq 2\exp\left(2\frac{\mu_M}{L_m}z \right)  z\left(\dfrac{\mu_M \wtL }{L_m^2}\| f_1-f_2\|+\dfrac{\wtmu}{L_m} \| f_1-f_2\|  \right)\leq D_1(z) \| f_1-f_2\|.
\end{align*}
In a similar way, by using the mean value theorem again  we get
\begin{align*}
&|I(f_1)(z)-I(f_2)(z)|\leq 2\exp\left(\frac{N_M}{L_m}z^2 \right) \bint_0^z  \sigma \left| \dfrac{N^*(f_1)(\sigma)}{L^*(f_1)(\sigma)}-\dfrac{N^*(f_2)(\sigma)}{L^*(f_2)(\sigma)} \right| \dd \sigma \\[0.15cm]
&\leq 2\exp\left(\frac{N_M}{L_m}z^2 \right) \left\lbrace \bint_0^z  \sigma  \left|  \tfrac{N^*(f_1)(\sigma)}{L^*(f_1)(\sigma) L^*(f_2)(\sigma)}\right|  \left| L^*(f_2)(\sigma)-L^*(f_1)(\sigma)\right| \dd \sigma \right.\\[0.15cm]
&+\left.\bint_0^z \sigma  \left|  \tfrac{1}{L^*(f_2)(\sigma) }\right|  \left| N^*(f_2)(\sigma)-N^*(f_1)(\sigma)\right| \dd \sigma\right\rbrace\\[0.15cm]
&\leq 2\exp\left(\frac{N_M}{L_m}z^2 \right)  \dfrac{z^2}{2} \left( \dfrac{N_M \wtL}{L_m^2} \|f_1-f_2\|+ \dfrac{\wtN}{L_m}\|f_1-f_2\|\right)\leq D_2(z) \|f_1-f_2\|.
\end{align*}

Taking into account that $E$  defined by \eqref{Def-E}  depends on $U$ and $I$ we  use \eqref{cota-U}, \eqref{cota-I} and the properties \eqref{Lipschitz-U} and \eqref{Lipschitz-I} that we have just proved   in order to get  the inequality \eqref{Lipschitz-E}.
\begin{align*}
&|E(f_1)(z)-E(f_2)(z)|\leq \tfrac{1}{|I(f_1)(z)||I(f_2)(z)|} \left| U(f_1)(z)I(f_2)(z)-I(f_1)(z)U(f_2)(z) \right|\\[0.15cm]
&\leq \exp\left(-2\tfrac{N_m}{L_M}z^2 \right) \Big\lbrace |U(f_1)(z)||I(f_2)(z)-I(f_1)(z)|\\[0.15cm]
&\qquad +|I(f_1)(z)||U(f_1)(z)-U(f_2)(z)| \Big\rbrace\\[0.15cm]
&\leq  \exp\left(2\tfrac{\mu_M}{L_m}z \right) D_2(z) \|f_1-f_2\|\\[0.15cm]
&\qquad+\exp\left( \dfrac{N_M}{L_m}z^2\right) D_1(z) \|f_1-f_2\|\leq D_3(z) \| f_1-f_2\|.
\end{align*}
Finally, by virtue of the definition \eqref{Def-Phi} of $\Phi$ and inequalities \eqref{cota-E} and \eqref{Lipschitz-E} we obtain 
\begin{align*}
&|\Phi(f_1)(z)-\Phi(f_2)(z)|\leq \bint_0^{z} \left\lbrace \dfrac{|E(f_1)(\eta)|}{|L^*(f_1)(\eta)L^*(f_2)(\eta)|} |L^*(f_2)(\eta)-L^*(f_1)(\eta)|\right.\\[0.15cm]
&+\left. \dfrac{1}{|L^*(f_2)(\eta)|} | E(f_1)(\eta)-E(f_2)(\eta) | \right\rbrace \dd \eta\\[0.15cm]
&\leq \bint_0^{z} \left\lbrace\tfrac{\exp\left(2\tfrac{\mu_M}{L_m}\eta \right)}{L_m^2} \wtL \| f_1-f_2\|+\tfrac{1}{L_m} D_3(\eta)\|f_1-f_2 \|\right\rbrace \dd \eta\\[0.15cm]
&\leq \bint_0^{\lambda} \left\lbrace\tfrac{\exp\left(2\tfrac{\mu_M}{L_m}\eta \right)}{L_m^2} \wtL \| f_1-f_2\|+\tfrac{1}{L_m} D_3(\eta)\|f_1-f_2 \|\right\rbrace \dd z\\[0.15cm]
&\leq \lambda \;  \left( \tfrac{\exp\left(2\tfrac{\mu_M}{L_m}\lambda\right)}{L_m^2} \wtL+ \tfrac{D_3(\lambda)}{L_m} \right)      \| f_1-f_2\|=\lambda\;D_4(\lambda)  \|f_1-f_2\|.
\end{align*}
\end{proof}
 We are able now to state the following theorem
 
 \begin{teo} \label{TeoExistenciaf-Temp} Suppose that \eqref{Hipotesis-L*}-\eqref{Hipotesis-mu*} hold and
 \be
 \dfrac{2L_M \wtL}{L_m^2}<1. \label{Hipotesis-ExtraContraccion}
 \ee 
 If $0<\lambda<\overline{\lambda}$ where $\overline{\lambda}>0$ is defined as the unique solution to {\color{black}$\mathcal{E}(z)=1$} with
 {\color{black}\be
 \mathcal{E}(z):=2  D_4(z) L_M \exp\left(\tfrac{N_M}{L_m}z^2 \right), \label{Epsilon-Temp}
 \ee}
 then there exists a unique solution $f\in C^0[0,\lambda]$ for the  integral equation \eqref{Volterra-Temp}, i.e. \eqref{PtoFijo-Temp}.

 \end{teo}
 
 \begin{proof}
As equation \eqref{Volterra-Temp} is equivalent to \eqref{PtoFijo-Temp}, we will show that $\mathcal{H}$ given by \eqref{Def-H} is a contracting self-map of $C^0[0,\lambda]$.

 On one hand, notice that for each $\lambda>0$, taking into account the definition of $\mathcal{H}$ and the hypothesis on $L^*, N^*$ and $\mu^*$ we can easily check that $\mathcal{H}$ maps $C^0[0,\lambda]$ onto itself. 
 
 On the other hand, let $f_1,f_2\in C^0[0,\lambda]$, from \eqref{cota-Phi} and \eqref{Lipschitz-Phi}, for each $0<\xi<\lambda$  we get
 \begin{align*}
 &|\mathcal{H}(f_1)(\xi)-\mathcal{H}(f_2)(\xi)|\leq \dfrac{|\Phi(f_1)(\xi)|}{|\Phi(f_1)(\lambda)| |\Phi(f_2)(\lambda)|} |\Phi(f_2)(\lambda)-\Phi(f_1)(\lambda)|\\[0.15cm]
 &+ \dfrac{1}{|\Phi(f_2)(\lambda)|} |\Phi(f_1)(\xi)-\Phi(f_2)(\xi)|\\[0.15cm]
 &\leq \dfrac{1}{|\Phi(f_2)(\lambda)|} \Big(|\Phi(f_2)(\lambda)-\Phi(f_1)(\lambda)|+|\Phi(f_2)(\xi)-\Phi(f_1)(\xi)|\Big)\\[0.15cm]
 &\leq \dfrac{L_M}{\lambda} \exp\left(\tfrac{N_M}{L_m}\lambda^2 \right) 2 \lambda D_4(\lambda)\| f_1-f_2\|.
 \end{align*}
 Therefore we obtain
$$\| \mathcal{H}(f_1)-\mathcal{H}(f_2)\| \leq \mathcal{E}(\lambda)\|f_1-f_2\|,$$
with $\mathcal{E}$ defined by \eqref{Epsilon-Temp}. Notice that $\mathcal{E}$ satisfies 
$$\mathcal{E}(0)=\dfrac{2L_M\wtL}{L_m^2},\qquad\qquad \mathcal{E}(+\infty)=+\infty,\qquad \qquad {\color{black}\mathcal{E}'(z)>0,\quad \forall z>0}.$$
Under the assumption \eqref{Hipotesis-ExtraContraccion} we deduce that there exists a unique $\overline{\lambda}>0$ such that $\mathcal{E}(\overline{\lambda})=1$. Moreover but most significantly we get
{\color{black}$$\mathcal{E}(z)<1,\qquad \forall \; 0<z<\overline{\lambda}\qquad\qquad \text{and}\qquad \mathcal{E}(z)>1,\qquad \forall \; z>\overline{\lambda}.$$}
In conclusion, if $\lambda$ is such that $0<\lambda<\overline{\lambda}$ then $\mathcal{E}(\lambda)<1$, and so $\mathcal{H}$ becomes a contraction mapping. By the fixed point Banach theorem we can say that there exists a unique solution $f\in C^0[0,\lambda]$ to the integral equation \eqref{PtoFijo-Temp}, i.e. to the integral equation \eqref{Volterra-Temp}.
 \end{proof}
 
 \begin{obs} The solution $f$ of \eqref{Volterra-Temp} depends implicitly on the positive number $\lambda$. This means that $f(\xi)=f_\lambda(\xi)=f(\xi,\lambda)$,\;$\forall\; 0<\xi<\lambda$.
 \end{obs}

 So far, we have proved, for a fixed $0<\lambda<\overline{\lambda}$, the existence of a unique solution to equation \eqref{Volterra-Temp}, which will be referred as $f_\lambda(\xi)$ in view of the dependence outlined in the prior remark.

It remains to analyse the existence of a solution $(f_{\widetilde{\lambda}},\widetilde{\lambda})$ to the system \eqref{Volterra-Temp}-\eqref{CondLambda-Temp}. 
So we will focus now on  condition \eqref{CondLambda-Temp}.

Let us define the function $\mathcal{V}(\lambda)=\mathcal{V}(f_\lambda,\lambda)$ as
\be
\mathcal{V}(\lambda):= \dfrac{\Ste}{2} \dfrac{E(f_\lambda)(\lambda)}{\Phi(f_{\lambda})(\lambda)},\qquad 0< \lambda<\overline{\lambda}.\label{Def-VLambda-Temp}
\ee
Then equation \eqref{CondLambda-Temp} is equivalent to
\be
\mathcal{V}(\lambda)=\lambda,\qquad 0<\lambda<\overline{\lambda}. \label{PtoFijo-Lambda-Temp}
\ee
The study of the equation \eqref{PtoFijo-Lambda-Temp} requires
 the following results:

\begin{lema} \label{LemaAcotaV-Temp}Assume \eqref{Hipotesis-L*}-\eqref{Hipotesis-mu*} and \eqref{Hipotesis-ExtraContraccion}.
Then for all $\lambda\in (0,\overline{\lambda})$ we have that
\be
\mathcal{V}_1(\lambda)<\mathcal{V}(\lambda)<\mathcal{V}_2(\lambda) \label{V1V2-Temp}
\ee
where $\mathcal{V}_1$ and $\mathcal{V}_2$ are functions defined by
\begin{eqnarray}
\begin{array}{lll}
&\mathcal{V}_1(\lambda)=\mathrm{Ste}\;\mu_M \exp\left(-2\lambda \tfrac{\mu_M}{L_m}-2\lambda^2 \tfrac{N_M}{L_m} \right),\qquad\; \quad &\lambda>0\\[0.15cm]
&\mathcal{V}_2(\lambda)=\dfrac{\mathrm{Ste}}{\sqrt{\pi}}\dfrac{\sqrt{N_M}}{\sqrt{L_m}}L_M \dfrac{\exp\left(2\lambda \tfrac{\mu_M}{L_m}-\lambda^2 \tfrac{N_m}{L_M} \right)}{\erf\left(\sqrt{\tfrac{N_M}{L_m}}\lambda \right)},\qquad &\lambda>0,
\end{array}
\end{eqnarray}
that satisfy the following properties
\begin{eqnarray}
\begin{array}{llll}
&\mathcal{V}_1(0)=\mathrm{Ste}\;\mu_M>0\quad  &\mathcal{V}_1(+\infty)=0,\quad  & \mathcal{V}_1'(\lambda)<0,\; \forall \lambda> 0,\\[0.15cm]
&\mathcal{V}_2(0)=+\infty, &\mathcal{V}_2(+\infty)=0. 
\end{array}
\end{eqnarray}

\end{lema}
\begin{proof}
Inequality \eqref{V1V2-Temp} arises immediately from \eqref{cota-E} and \eqref{cota-Phi}.

The properties for $\mathcal{V}_1$ can be easily checked by its own definition.

Let us analyse the function $\mathcal{V}_2$, on one hand, when $\lambda\to 0$, we have $\erf\left(\sqrt{\tfrac{N_M}{L_m}}\lambda\right)\to 0$ and $\exp\left(2\lambda \tfrac{\mu_M}{L_m}-\lambda^2 \tfrac{N_m}{L_M} \right)\to 1$ so we obtain that $\mathcal{V}_2(0)=+\infty$.

In a similar manner, when $\lambda\to+\infty$ we have  $\erf\left(\sqrt{\tfrac{N_M}{L_m}}\lambda\right)\to 1$ and $\exp\left(2\lambda \tfrac{\mu_M}{L_m}-\lambda^2 \tfrac{N_m}{L_M} \right)\to 0$, so that $\mathcal{V}_2(+\infty)= 0$.


\end{proof}

\medskip

\begin{lema}\label{LemaLambda12-Temperatura} There exists a unique solution $\lambda_1$ 
to the equation  
\be
\mathcal{V}_1(\lambda)=\lambda,\quad \lambda> 0,\label{Lambda1-temp}
\ee

and there exists at least one solution $\lambda_2>\lambda_1$ to the equation
\be
\mathcal{V}_2(\lambda)=\lambda,\quad \lambda>0.\label{Lambda2-temp}
\ee
\end{lema}
\begin{proof}
The proof is straightforward taking into account the properties of $\mathcal{V}_1$ and $\mathcal{V}_2$ presented in Lemma \ref{LemaAcotaV-Temp} which are independent of the hypothesis.
\end{proof}

\medskip

\begin{teo} \label{TeoCotaTemp} Assume that \eqref{Hipotesis-L*}-\eqref{Hipotesis-mu*} and \eqref{Hipotesis-ExtraContraccion}  hold.  Consider $\lambda_1$ and $\lambda_2$ given by \eqref{Lambda1-temp} and \eqref{Lambda2-temp}, respectively.
If $\mathcal{E}(\lambda_2)<1$ 
  then there exists at least one solution $\widetilde{\lambda}\in (\lambda_1,\lambda_2)$ to the equation \eqref{PtoFijo-Lambda-Temp}.
\end{teo}
\begin{proof}

Under the hypothesis of Lemma \ref{LemaAcotaV-Temp}, if $\mathcal{E}(\lambda_2)<1$  we have that for each $\lambda_1\leq \lambda\leq  \lambda_2<\overline{\lambda}$ the inequality \eqref{V1V2-Temp} holds and $\mathcal{E}(\lambda)<1$.
As $\mathcal{V}$ is a continuous function we obtain that there exists at least one solution $\widetilde{\lambda}$ to the equation $\mathcal{V}(\lambda)=\lambda$ that belongs to the  interval $(\lambda_1,\lambda_2)$.

\end{proof}

We have found  sufficient conditions  that guarantee the existence of at least one solution $(f_{\widetilde{\lambda}},\widetilde{\lambda})$ to the 
the problem \eqref{Volterra-Temp}-\eqref{CondLambda-Temp}. Let us return the original problem \eqref{Ec:Calor}-\eqref{Cond:FronteraInicial}.

From the prior analysis, and taking into account that condition \eqref{Hipotesis-ExtraContraccion} can be rewritten as
\be\label{Hipotesis-Epsilon-Temp}
\dfrac{2 k_M \wtk (T^*-T_m)}{k_m^2}<1,
\ee
 we are able to state the main result of this section.

\begin{teo}\label{teotemp} Assume that $k$, $\rho$, $c$ and $v$ are such that  \eqref{Hipotesis-k}, \eqref{Hipotesis-rhoc}, \eqref{Hipotesis-v} and  \eqref{Hipotesis-Epsilon-Temp} hold. If $\mathcal{E}(\lambda_2)<1$ where $\mathcal{E}$ is given by \eqref{Epsilon-Temp} and $\lambda_2$ is defined by \eqref{Lambda2-temp}, then there exists at least one solution to the Stefan problem \eqref{Ec:Calor}-\eqref{Cond:FronteraInicial}, where the free boundary is given by
\be
s(t)=2\widetilde{\lambda}\sqrt{\alpha_0 t} ,\qquad t>0,
\ee
with $\widetilde{\lambda}$ defined by Theorem \ref{TeoCotaTemp}, and the temperature is given by
\be
T(x,t)=(T_m-T^*)f_{\widetilde{\lambda}}(\xi)+T^* ,\qquad \qquad {\color{black}0\leq \xi\leq \widetilde{\lambda} }
\ee
being $\xi=\frac{x}{2\sqrt{\alpha_0 t}}$ the similarity variable and $f_{\widetilde{\lambda}}$ the unique solution of the integral equation \eqref{Volterra-Temp} which was established in Theorem \ref{TeoExistenciaf-Temp}.
\end{teo}

\medskip

\section{Neumann condition}

In this section we will study the Stefan problem that arises when we consider a Neumann condition at the fixed face.

Let us notice that in case we consider the Neumann  condition (\ref{Cond:Flujo x=0}), instead of the Dirichlet condition (\ref{Cond:Temp x=0}), after the change of  variables 
\be
\theta=\dfrac{T-T_m}{T_m}\qquad \qquad (T(\theta)=T_m \theta+T_m)
\ee
the problem governed by  (\ref{Ec:Calor}), (\ref{Cond:Flujo x=0}), (\ref{Cond:TempCambioFase})-(\ref{Cond:FronteraInicial}) 
becomes
\begin{subequations}
\begin{align}
& \Nr(\theta) \pder[\theta]{t}=\alpha_0 \frac{\partial}{\partial x}\left( \Lr(\theta)\pder[\theta]{x} \right)-\vr(\theta) \pder[\theta]{x}
,& 0<x<s(t), \quad t>0, \label{Ec:Calor-Theta-Flujo}\\
& \Lr(\theta(0,t)) \pder[\theta]{x}(0,t)=\dfrac{-q}{k_0 T_m \sqrt{t}}, &t>0, \label{Cond:Flujo x=0-Theta}\\
&  \theta(s(t),t)=0, &t>0, \label{Cond:TempCambioFase-Theta-Flujo}\\
&\Lr(\theta(s(t),t))\pder[\theta]{x}(s(t),t)=\dfrac{-\rho_0 \ell}{k_0 T_m}\dot{s}(t) , &t>0, \label{Cond:Stefan-Theta-Flujo}\\
& s(0)=0,\label{Cond:FronteraInicial-Theta-Flujo}
\end{align}
\end{subequations}
where $\Nr,\Lr$ and $\vr$ are given by \eqref{LNv-Raya} with $T(\theta)=T_m\theta+T_m$.

Then, if we introduce the similarity transformation (\ref{similarity-variable}), we obtain the following ordinary differential problem
\begin{subequations}
\begin{align}
&\Big(L^*(f)f'(\xi)\Big)'+2f'(\xi) \Big(N^*(f)\xi-\mu^*(f) \Big)=0 ,& 0< \xi<\lambda \label{EcDif-f-flujo}\\
&L^*(f(0))f'(0)=-q^*,  \label{Condf=0-flujo}\\
&  f(\lambda)=0,  \label{Condf=1-flujo}\\
& f'(\lambda)=-M\lambda, \label{EcLambda-flujo}
\end{align}
\end{subequations}
where $\xi=\tfrac{x}{2\sqrt{\alpha_0 t}}$ is the similarity variable and  $L^*$, $N^*$ and $\mu^*$ are given by 
\eqref{LNv-Estrella} and \eqref{Def:muEstrella}, respectively. Moreover, $q^*$ and $M$ are defined by $q^*=\tfrac{2q\sqrt{\alpha_0}}{k_0 T_m}$ and \mbox{$M=\tfrac{2\ell k_0}{T_m c_0 \;k(T_m)}$.}

\medskip

We can deduce that $(f,\lambda)$ is a solution to the ordinary differential problem \eqref{EcDif-f-flujo}-\eqref{EcLambda-flujo} if and only if $(f,\lambda)$ satisfies the following integral equation
\be
f(\xi) =q^*\left(\Phi(f)(\lambda)-\Phi(f)(\xi) \right),\qquad {\color{black}0\leq \xi\leq \lambda},\label{EcIntegral-Flujo}
\ee
togehter with the condition 
\be
f'(\lambda)=-M\lambda \label{CondLambda-Flujo}
\ee
where $\Phi(f)$ is defined by \eqref{Def-Phi}.
\medskip

In order to analyse the existence of solution to this problem, let us study first, for a fixed  $\lambda>0$
the integral equation \eqref{EcIntegral-Flujo} for $f$.

In the same manner as we did in the first section, we consider the space $C^0[0,\lambda]$ of continuous real-valued functions defined on $[0,\lambda]$  endowed with the supremum norm and we define the operator  $\mathcal{H}^q$ on $C^0[0,\lambda]$ given by the r.h.s of equation \eqref{EcIntegral-Flujo}:
\be  
\mathcal{H}^q(f)(\xi)=q^*\Big( \Phi(f)(\lambda)-\Phi(f)(\xi)\Big).\label{Def-Hq}
\ee

{


Then, as $(C^0[0,\lambda], \Vert\cdot\Vert)$ is a Banach space we  use the fixed point Banach theorem to prove that for each $\lambda>0$ there exists a unique $f$ such that
\be
\mathcal{H}^q(f)(\xi)=f(\xi) ,\qquad 0\leq \xi\leq \lambda,\label{PtoFijo-Flujo}
\ee
which is  the solution  to \eqref{EcIntegral-Flujo}.

If we assume that $k$, $\rho$, $c$ and $v$ satisfy \eqref{Hipotesis-k}, \eqref{Hipotesis-rhoc} and \eqref{Hipotesis-v}, respectively then $L^*$, $N^*$ and $\mu^*$ verify \eqref{Hipotesis-L*}, \eqref{Hipotesis-N*} and \eqref{Hipotesis-mu*} where in this case 
\be\label{LNmutilde-Flujo}
\widetilde{L}=\tfrac{\widetilde{k}|T_m|}{k_0},\qquad\qquad \widetilde{N}=\tfrac{\widetilde{\gamma} |T_m|}{\rho_0 c_0}\qquad\qquad \widetilde{\mu}=\tfrac{\widetilde{\nu}|T_m|}{\sqrt{\rho_0 c_0 k_0}}.
\ee

Therefore we are able to use the bounds  in Lemma \ref{LemaCotasUIEPhi} and the Lipschitz continuities obtained  in Lemma \ref{LemaLipschitzUIEPhi}. Hence, we can  state the following theorem
 
 \begin{teo} \label{TeoExistenciaf-Flujo} Suppose that $L^*$, $N^*$ and $\mu^*$ satisfy conditions \eqref{Hipotesis-L*}-\eqref{Hipotesis-mu*}. 
 If $0<\lambda<\overline{\lambda_q}$ where $\overline{\lambda_q}>0$ is defined as the unique solution to $\mathcal{E}_q(z)=1$ with
 \be
 \mathcal{E}_q(z):=2 q^* z D_4(z), \label{Epsilonq-Flujo}
 \ee
and $D_4$ is given by \eqref{D1-4},  then there exists a unique solution $f\in C^0[0,\lambda]$ for the  integral equation \eqref{EcIntegral-Flujo}, i.e. \eqref{PtoFijo-Flujo}.

 \end{teo}

  \begin{proof}
As solving equation \eqref{EcIntegral-Flujo} is equivalent to find a fixed point to the operator  $\mathcal{H}^q$ given by \eqref{Def-Hq}, we will show that it is a contracting self-map of $C^0[0,\lambda]$.

 On one hand, notice that for each $\lambda>0$, taking into account the definition of $\mathcal{H}^q$ and the hypothesis on $L^*, N^*$ and $\mu^*$ we can easily check that $\mathcal{H}^q$ maps $C^0[0,\lambda]$ onto itself. 
 
 On the other hand, let $f_1,f_2\in C^0[0,\lambda]$, from  \eqref{Lipschitz-Phi}, for each $0\leq\xi\leq\lambda$  we get
 \begin{align*}
 &|\mathcal{H}^q(f_1)(\xi)-\mathcal{H}^q(f_2)(\xi)|\leq  \left|q^*\Big(\Phi(f_1)(\lambda)-\Phi(f_1)(\xi)\Big)-q^*\Big(\Phi(f_2)(\lambda)-\Phi(f_2)(\xi)\Big) \right|\\[0.15cm]
&\leq q^* \Big(|\Phi(f_1)(\lambda)-\Phi(f_2)(\lambda)|+ |\Phi(f_1)(\xi)-\Phi(f_2)(\xi)|\Big)\leq 2 q^* \lambda D_4(\lambda) \| f_1-f_2\|.
 \end{align*}
 Therefore it follows that
$$\| \mathcal{H}^q(f_1)-\mathcal{H}^q(f_2)\| \leq \mathcal{E}_q(\lambda)\|f_1-f_2\|,$$
where $\mathcal{E}_q$ defined by \eqref{Epsilonq-Flujo}, is an increasing function that goes from 0 to $+\infty$ when $z$ goes from 0 to $+\infty$.
Thus there exists a unique $\overline{\lambda}_q>0$ such that $\mathcal{E}_q(\overline{\lambda}_q)=1$. Moreover but most significantly we get
$$\mathcal{E}_q(z)<1,\qquad \forall \; 0<z<\overline{\lambda_q}\qquad\qquad \text{and}\qquad \mathcal{E}_q(z)>1,\qquad \forall \; z>\overline{\lambda_q}.$$
Then, if $\lambda$ is such that $0<\lambda<\overline{\lambda}_q$ we get that $\mathcal{E}_q(\lambda)<1$ and so the operator $\mathcal{H}^q$ becomes a contraction mapping. By the fixed point Banach theorem it must exists a unique solution $f\in C^0[0,\lambda]$ to the integral equation \eqref{PtoFijo-Flujo}, i.e. to the integral equation \eqref{EcIntegral-Flujo}.
 \end{proof}

For each $0<\lambda<\overline{\lambda}_q$ fixed, we have a  unique solution to equation \eqref{EcIntegral-Flujo}, which will be referred as $f(\xi)=f_\lambda(\xi)$ to make visible its dependence on $\lambda$. Notice that we have
\be
f_{\lambda}'(\xi)=-q^* \dfrac{E(f_\lambda)(\xi)}{L^*(f_\lambda)(\xi)}.
\ee
Then the condition \eqref{CondLambda-Flujo}, which remains to be analysed,  becomes equivalent to
\be
\mathcal{V}^q(\lambda)=\lambda,\label{PtoFijoLambda-Flujo}
\ee
where
\be 
\mathcal{V}^q(\lambda)=\mathcal{V}^q({f_\lambda},\lambda):= \dfrac{q^* }{M L^*(f_\lambda)(\lambda)} E_{f_\lambda}(\lambda).\label{Def-Vq-Flujo}
\ee

\medskip

\begin{lema} \label{LemaAcotaVq-Flujo}Assume that \eqref{Hipotesis-L*}-\eqref{Hipotesis-mu*} hold.
Then for all $\lambda\in (0,\overline{\lambda}_q)$ we have that
\be
\mathcal{V}_1^q(\lambda)<\mathcal{V}^q(\lambda)<\mathcal{V}_2^q(\lambda) \label{V1qV2q-Flujo}
\ee
where $\mathcal{V}^q_1$ and $\mathcal{V}^q_2$ are functions defined by
\begin{eqnarray}
\begin{array}{lll}
&\mathcal{V}^q_1(\lambda)=\dfrac{q^*}{M L_M} \exp\left(-\lambda^2\tfrac{N_M}{L_M} \right),\qquad\; \quad &\lambda>0\\ \\
&\mathcal{V}^q_2(\lambda)=\dfrac{q^*}{M L_m}\exp\left(2\lambda\tfrac{\mu_M}{L_m}-\lambda^2\tfrac{N_M}{L_m} \right) ,\qquad &\lambda>0,
\end{array}
\end{eqnarray}
that satisfy the following properties:
\begin{eqnarray}
\begin{array}{llll}
&{\mathcal{V}}^q_1(0)=\dfrac{q^*}{M L_M}>0, \qquad {\mathcal{V}^q_1}(+\infty)=0, \qquad{\mathcal{V}^q_1}'(\lambda)<0,\;\forall \lambda> 0,\\ \\
&\mathcal{V}^q_2(0)=\dfrac{q^*}{M L_m}>0,\qquad \mathcal{V}^q_2(+\infty)=0,\qquad\qquad \\ \\
&{\mathcal{V}^q_2}'(\lambda)>0,\qquad \forall\; 0< \lambda< \tfrac{\mu_M L_M}{L_m N_m},\qquad {\mathcal{V}_2^q}'(\lambda)<0,\qquad \forall  \lambda> \tfrac{\mu_M L_M}{L_m N_m} .
\end{array}
\end{eqnarray}
\end{lema}

\begin{proof}
Inequality \eqref{V1qV2q-Flujo}  follows directly  from the bounds obtained in \eqref{cota-E}.
The properties for  $\mathcal{V}^q_1$ and $\mathcal{V}^q_2$ arises straightforward from its definitions.
\end{proof}

\begin{lema} There exists a unique solution $\lambda_{1q}$ 
to the equation  
\be
\mathcal{V}^q_1(\lambda)=\lambda,\quad \lambda> 0,\label{Lambda1-Flujo}
\ee
and there exists a unique solution $\lambda_{2q}>\lambda_{1q}$ to the equation
\be
\mathcal{V}^q_2(\lambda)=\lambda,\quad \lambda>0.\label{Lambda2-Flujo}
\ee
\end{lema}
\begin{proof}
It is immediate taking into account the properties of $\mathcal{V}^q_1$ and $\mathcal{V}^q_2$ shown in Lemma \ref{LemaAcotaVq-Flujo}.
\end{proof}

\begin{teo} \label{TeoCota-Flujo} Assume that \eqref{Hipotesis-L*}-\eqref{Hipotesis-mu*} hold.  Consider $\lambda_{1q}$ and $\lambda_{2q}$ given by \eqref{Lambda1-Flujo} and \eqref{Lambda2-Flujo}, respectively.
If $\mathcal{E}_q(\lambda_{2q})<1$, where $\mathcal{E}_q$ is defined by \eqref{Epsilonq-Flujo},
  then there exists at least one solution $\widetilde{\lambda}_q\in (\lambda_{1q},\lambda_{2q})$ to the equation \eqref{PtoFijoLambda-Flujo}.
\end{teo}
\begin{proof}

It is similar to the proof given in Theorem \ref{TeoCotaTemp}.


\end{proof}

We have found  sufficient conditions  that guarantee the existence of at least one solution $(f_{\widetilde{\lambda}},\widetilde{\lambda})$ to the 
the problem \eqref{EcIntegral-Flujo}-\eqref{CondLambda-Flujo}. Let us return the original problem \eqref{Ec:Calor}, \eqref{Cond:Flujo x=0}, \eqref{Cond:Stefan}-\eqref{Cond:FronteraInicial}. After the prior analysis we are able to state the following result.

\begin{teo} Assume that $k$, $\rho$, $c$ and $v$ are such that  \eqref{Hipotesis-k}, \eqref{Hipotesis-rhoc} and  \eqref{Hipotesis-v}  hold. If $\mathcal{E}_q(\lambda_{2q})<1$ where $\mathcal{E}_q$ is given by \eqref{Epsilonq-Flujo} and $\lambda_{2q}$ is defined by \eqref{Lambda2-Flujo}, then there exists at least one solution to the Stefan problem \eqref{Ec:Calor}, \eqref{Cond:Flujo x=0}, \eqref{Cond:Stefan}-\eqref{Cond:FronteraInicial}, where the free boundary is given by
\be
s(t)=2\widetilde{\lambda}_q\sqrt{\alpha_0 t} ,\qquad t>0,
\ee
with $\widetilde{\lambda}_q$ defined by Theorem \ref{TeoCota-Flujo}, and the temperature is given by
\be
T(x,t)=T_m f_{\widetilde{\lambda_q}}(\xi)+T_m ,\qquad \qquad 0\leq\xi\leq\widetilde{\lambda}_q
\ee
being $\xi=\tfrac{x}{2\sqrt{\alpha_0 t}}$ the similarity variable and $f_{\widetilde{\lambda}_q}$ the unique solution of the integral equation \eqref{EcIntegral-Flujo} which was established in Theorem \ref{TeoExistenciaf-Flujo}.
\end{teo}

\section{Robin condition}

The following section is devoted to the analysis of the Stefan problem  that arises when we consider a Robin condition at the fixed face $x=0$.

Let us consider the problem \eqref{Ec:Calor}, (\ref{Cond:Convectiva x=0}), \eqref{Cond:TempCambioFase}-\eqref{Cond:FronteraInicial}.  After introducing the change of variables (\ref{cambio-theta}) we get
\begin{subequations}
\begin{align}
& \Nr(\theta) \pder[\theta]{t}=\alpha_0 \frac{\partial}{\partial x}\left( \Lr(\theta)\pder[\theta]{x} \right)-\vr(\theta) \pder[\theta]{x}
,& 0<x<s(t), \quad t>0, \label{Ec:Calor-Theta-Convectiva}\\
& \Lr(\theta(0,t)) \pder[\theta]{x}(0,t)=\dfrac{h}{k_0\sqrt{t}}\theta(0,t), &t>0, \label{Cond:Flujo x=0-Convectiva}\\
&  \theta(s(t),t)=0, &t>0, \label{Cond:TempCambioFase-Theta-Convectiva}\\
&\Lr(\theta(s(t),t))\pder[\theta]{x}(s(t),t)=\dfrac{\dot{s}(t)}{\alpha_0 \Ste} , &t>0, \label{Cond:Stefan-Theta-Convectiva}\\
& s(0)=0,\label{Cond:FronteraInicial-Theta-Convectiva}
\end{align}
\end{subequations}
where $\Nr,\Lr$ and $\vr$ are given by \eqref{LNv-Raya} with $T(\theta)=(T_m-T^*)\theta+T^*$.

The similarity transformation (\ref{similarity-variable}) yields to an ordinary differential problem defined by 
\begin{subequations}
\begin{align}
&\Big(L^*(f)f'(\xi)\Big)'+2f'(\xi)\Big( N^*(f)\xi-\mu^*(f)\Big)=0 ,& 0<\xi<\lambda \label{EcDif-f-convectiva}\\
&L^*(f(0))f'(0)=2\Bi\; f(0),  \label{Condf=0-convectiva}\\
&  f(\lambda)=1,  \label{Condf=1-convectiva}\\
&L^*(f(\lambda)) f'(\lambda)=\dfrac{2\lambda}{\Ste}, \label{EcLambda-convectiva}
\end{align}
\end{subequations}
where $\xi=\tfrac{x}{2\sqrt{\alpha_0 t}}$ is the similarity variable,  $L^*$, $N^*$ and $\mu^*$ are given by \eqref{LNv-Estrella} and \eqref{Def:muEstrella}, respectively and $\Bi=\tfrac{h \sqrt{\alpha_0} }{k_0}$  is the Biot number.
 
In addition, this ordinary differential problem is equivalent to find $(f,\lambda)$ such that the integral equation:
\be\label{EcIntegral-Convectiva}
f(\xi)=\dfrac{1+2\Bi \;\Phi(f)(\xi)}{1+2\Bi\;\Phi(f)(\lambda)},\qquad 0\leq\xi\leq\lambda, 
\ee
together with the condition
\be\label{CondfLambda-Convectivo}
f'(\lambda)={\color{black}\frac{2}{L^*(f)(\lambda)\Ste}\lambda}
\ee
 hold, where $\Phi(f)$ is given by \eqref{Def-Phi}.
 
 Let us consider a fixed $\lambda>0$, then we can define the operator $\mathcal{H}^h$ on $C^0[0,\lambda]$ as
 \be \label{OperadorHh-Convectiva}
 \mathcal{H}^h(f)(\xi)=\dfrac{1+2\Bi \;\Phi(f)(\xi)}{1+2\Bi\;\Phi(f)(\lambda)}.
\ee
Therefore,  the integral equation \eqref{EcIntegral-Convectiva} can be rewritten as
\be\label{PtoFijo-Convectiva}
\mathcal{H}^h(f)(\xi)=f(\xi),\qquad \qquad 0\leq\xi\leq\lambda. 
\ee

Let us assume that $k$, $\rho$, $c$ and $v$ satisfy \eqref{Hipotesis-k}, \eqref{Hipotesis-rhoc} and \eqref{Hipotesis-v}, respectively then $L^*$, $N^*$ and $\mu^*$ verify \eqref{Hipotesis-L*}, \eqref{Hipotesis-N*} and \eqref{Hipotesis-mu*}.
As a consequence, we can state the following results.
\medskip
 \begin{teo} \label{TeoExistenciaf-Convectiva} Suppose that $L^*$, $N^*$ and $\mu^*$ satisfy conditions \eqref{Hipotesis-L*}-\eqref{Hipotesis-mu*} and \eqref{Hipotesis-ExtraContraccion}.
 If $0<\lambda<\overline{\lambda}_h$ where $\overline{\lambda}_h>0$ is defined as the unique solution to $\mathcal{E}_h(z)=1$ with
 \be
 \mathcal{E}_h(z):=2 L_M D_4(z) \exp\left( z^2 \tfrac{N_M}{L_m}\right), \label{Epsilonh-Convectiva}
 \ee
and $D_4$ is given by \eqref{D1-4},  then there exists a unique solution $f\in C^0[0,\lambda]$ for the  integral equation \eqref{EcIntegral-Convectiva}, i.e. \eqref{PtoFijo-Convectiva}.

 \end{teo}

 \begin{obs}
Notice that $\mathcal{E}_h=\mathcal{E}$ where $\mathcal{E}$ was defined in Theorem \ref{TeoExistenciaf-Temp}, when we studied the problem with  a Dirichlet condition at the fixed face. Hence, we have that $\overline{\lambda}_h=\overline{\lambda}$.
 \end{obs}

We have obtained, for each $0<\lambda<\overline{\lambda}_h$ fixed,  a  unique solution to equation \eqref{EcIntegral-Flujo} , $f(\xi)=f_\lambda(\xi)$. If we compute its derivative, we get
\be
f_{\lambda}'(\xi)=\dfrac{2\Bi}{\left(1+2\Bi \; \Phi({f_\lambda})(\lambda)\right)}\dfrac{E(f_\lambda)(\xi)}{L^*(f_\lambda)(\xi)} . 
\ee
Then the condition \eqref{CondfLambda-Convectivo}, which remains to be analysed,  becomes equivalent to
\be
\mathcal{V}^h(\lambda)=\lambda,\label{PtoFijoLambda-Convectivo}
\ee
where
\be 
\mathcal{V}^h(\lambda)=\mathcal{V}^h({f_\lambda},\lambda):= \dfrac{\Ste\; \Bi\; E(f_\lambda)(\lambda)}{1+2\Bi\; \Phi(f_\lambda)(\lambda) }.\label{Def-Vh-Convectivo}
\ee
We have the following results
\medskip

\begin{lema} \label{LemaAcotaVh-Convectivo} Assume that \eqref{Hipotesis-L*}-\eqref{Hipotesis-mu*} and \eqref{Hipotesis-ExtraContraccion} hold.
Then for all $\lambda\in (0,\overline{\lambda}_h)$ we have that
\be
0<\mathcal{V}^h(\lambda)<\mathcal{V}_2(\lambda), \label{V1hV2h-Convectivo}
\ee
where  $\mathcal{V}_2$ is given by \eqref{V1V2-Temp}.

\end{lema}

\begin{proof}
The proof follows straightforward by taking into account the bounds given in Lemma \ref{LemaCotasUIEPhi}. 
\end{proof}

Let us notice that, due to  the properties of $\mathcal{V}_2$ studied in Lemma \ref{LemaLambda12-Temperatura}, we know that there exists at least one solution $\lambda_{2}$ 
to the equation  
\be
\mathcal{V}_2(\lambda)=\lambda,\quad \lambda>0.\label{Lambda2-Convectivo}
\ee

}
\begin{teo} \label{TeoCota-Convectivo} Assume that \eqref{Hipotesis-L*}-\eqref{Hipotesis-mu*} and \eqref{Hipotesis-ExtraContraccion} hold.  Consider {\color{black}$\lambda_{2}$} given by \eqref{Lambda2-Convectivo}.
If $\mathcal{E}_h(\lambda_{2})<1$, where $\mathcal{E}_h$ is defined by \eqref{Epsilonh-Convectiva},
  then there exists at least one solution $\widetilde{\lambda}_h\in (0,\lambda_{2})$ to the equation \eqref{PtoFijoLambda-Convectivo}.
\end{teo}
\begin{proof}

It is similar to the proof given in Theorem \ref{TeoCotaTemp}. 


\end{proof}

 Let us return the original problem \eqref{Ec:Calor}, \eqref{Cond:Convectiva x=0}, \eqref{Cond:Stefan}-\eqref{Cond:FronteraInicial}. Notice that condition \eqref{Hipotesis-ExtraContraccion} can be rewritten as
 \be\label{Hipotesis-Extra-krhoc-Convectivo}
  \dfrac{2 k_M \wtk (T^*-T_m)}{k_m^2}<1,
 \ee
then we state the following main theorem.
\begin{teo} Assume that $k$, $\rho$, $c$ and $v$ are such that  \eqref{Hipotesis-k}, \eqref{Hipotesis-rhoc}, \eqref{Hipotesis-v}  and \eqref{Hipotesis-Extra-krhoc-Convectivo} hold. If $\mathcal{E}_h(\lambda_{2})<1$ where $\mathcal{E}_h$ is given by \eqref{Epsilonh-Convectiva} and {\color{black}$\lambda_{2}$} is given  by \eqref{Lambda2-Convectivo}, then there exists at least one solution to the Stefan problem \eqref{Ec:Calor}, \eqref{Cond:Convectiva x=0}, \eqref{Cond:Stefan}-\eqref{Cond:FronteraInicial}, where the free boundary is given by
\be
s(t)=2\widetilde{\lambda}_h\sqrt{\alpha_0 t} ,\qquad t>0,
\ee
with $\widetilde{\lambda}_h$ defined by Theorem \ref{TeoCota-Convectivo}, and the temperature is given by
\be
T(x,t)=(T_m-T^*)f_{\widetilde{\lambda}_h}(\xi)+T^* ,\qquad \qquad {\color{black}0\leq \xi\leq \widetilde{\lambda}_h }
\ee
being $\xi=\tfrac{x}{2\sqrt{\alpha_0 t}}$ the similarity variable and $f_{\widetilde{\lambda}_h}$ the unique solution of the integral equation \eqref{EcIntegral-Convectiva} which was established in Theorem \ref{TeoExistenciaf-Convectiva}.
\end{teo}

\begin{obs}
When the coefficient $h\to+\infty$, i.e. $Bi\to +\infty$ the integral equation \eqref{EcIntegral-Convectiva} becomes $\eqref{Volterra-Temp}$ and the equation \eqref{PtoFijoLambda-Convectivo} becomes \eqref{CondLambda-Temp}. Then,  the solution given by Theorem $\ref{TeoExistenciaf-Convectiva}$ for the problem with a Robin condition at the fixed face  converges to the solution to the problem with a Dirichlet condition at $x=0$ given by  Theorem $\ref{TeoExistenciaf-Temp}$ when $h\to\infty$.
\end{obs}

}

{\color{black} 
\begin{obs} If we consider that the  speed of the convective term in the heat equation \eqref{Ec:Calor} is  $v(T)\equiv 0$, we can recover  the solution obtained  in \cite{BrNa2019} for a null control function which is obtained as a solution to the following integral equation 
\be
f(\xi)=\dfrac{L^{*}(f(0))+2\mathrm{Bi} \;\Phi_{0}(f)(\xi)}{L^{*}(f(0))+2\mathrm{Bi}\;\Phi_{0}(f)(\lambda)}=\dfrac{1+2\mathrm{Bi} \;\Phi(f)(\xi)}{1+2\mathrm{Bi}\;\Phi(f)(\lambda)},\qquad 0\leq\xi\leq\lambda, 
\ee
where 
\be\Phi_{0}(f)(\xi)=\bint_0^{\xi} \dfrac{L^{*}(f(0))}{L^*(f)(z)I(f)(z)}\; \dd z=L^{*}(f(0))\bint_0^{\xi}  \dfrac{E(f)(z)}{L^*(f(z))}\; \dd z=L^{*}(f(0))\Phi(f)(\xi),
\ee and
\be
E(f)(z)=\dfrac{1}{I(f)(z)}
\ee
together with the condition for $\lambda$ given by
\be\label{CondfLambda-Convectivo}
f'(\lambda)={\color{black}\dfrac{2}{L^*(f)(\lambda)\mathrm{Ste}}\lambda}.
\ee
\end{obs}

}

\section{Radiative-convective condition}

We will proceed  to the analysis of the Stefan problem  that arises when we assume a radiative-convective condition at the fixed face $x=0$. 

If we consider the problem \eqref{Ec:Calor}, (\ref{Cond:RadiactivaConvectiva x=0}), \eqref{Cond:TempCambioFase}-\eqref{Cond:FronteraInicial}.  After the change of variables (\ref{cambio-theta}) and  introducing the similarity transformation (\ref{similarity-variable}) we obtain the following ordinary differential problem defined by 
\begin{subequations}
\begin{align}
&\Big(L^*(f)f'(\xi)\Big)'+2f'(\xi)\Big( N^*(f)\xi-\mu^*(f)\Big)=0 ,& 0<\xi<\lambda \label{EcDif-f-radiactiva}\\
&L^*(f(0))f'(0)=2\Bi\; f(0)+r F(f)(0),  \label{Condf=0-radiactiva}\\
&  f(\lambda)=1,  \label{Condf=1-radiactiva}\\
&L^*(f(\lambda)) f'(\lambda)=\dfrac{2\lambda}{\Ste}, \label{EcLambda-radiactiva}
\end{align}
\end{subequations}
where $\xi=\tfrac{x}{2\sqrt{\alpha_0 t}}$ is the similarity variable,  $L^*$, $N^*$ and $\mu^*$ are given by \eqref{LNv-Estrella} and \eqref{Def:muEstrella}, respectively and where $r=\dfrac{2\sigma \epsilon \sqrt{\alpha_0}}{k_0(T^*-T_m)}>0$ and
$F(f)(0)={T^*}^4-\Big((T_m-T^*)f(0)+T^* \Big)^4$.
 \medskip
 
It is easy to see that this ordinary differential problem is equivalent to find $(f,\lambda)$ such that the integral equation hold:
\be\label{EcIntegral-Radiactiva}
\mathcal{H}^r(f)(\xi)=f(\xi),\qquad 0\leq \xi\leq \lambda, 
\ee
together with the condition
\be\label{CondfLambda-Radiactiva}
f'(\lambda)={\color{black} \dfrac{2}{L^*(f(\lambda))\Ste} \lambda},
\ee
where the operator $\mathcal{H}^{r}$ on $C^0[0,\lambda]$ is defined by
\be
\mathcal{H}^r(f)(\xi):= 1-G(f)(0)\Big( \Phi(f)(\lambda)-\Phi(f)(\xi)\Big).
\ee
with $G(f)(0)=2\Bi f(0)+r F(f)(0)$
and  $\Phi(f)$ is given by \eqref{Def-Phi}.
 \medskip

 In order to solve the fixed point equation \eqref{EcIntegral-Radiactiva} for a fixed $\lambda>0$, let us consider the 
 set $X$ given by all non-negative functions  bounded by 1, i.e
 \be
 X=\lbrace f\in C^0[0,\lambda]: f\geq 0,\; ||f||\leq 1\rbrace.
 \ee
 Notice that $X$ is a non-empty closed subset of the Banach space $\left(C^0[0,\lambda],||\cdot || \right)$.

  \begin{teo} \label{TeoExistenciaf-Radiactiva} Suppose that $L^*$, $N^*$ and $\mu^*$ satisfy conditions \eqref{Hipotesis-L*}-\eqref{Hipotesis-mu*} as well as
  \be
 \dfrac{2\mathrm{Bi}+r {T^*}^4}{L_m\sqrt{\tfrac{N_m}{L_M}}}\sqrt{\pi}\exp\left(\tfrac{\mu_M^2 L_M}{L_m^2 N_m} \right)\leq 1, \label{Hipotesis:Contraccion-radiactiva}
 \ee 
 and
 \be
 \dfrac{2\mathrm{Bi}+r D_5}{\mu_M} <1.\label{Hipotesis:ExtraContraccion-radiactiva}
 \ee
 If $0<\lambda<\overline{\lambda}_r$ where $\overline{\lambda}_r>0$ is defined as the unique solution to $\mathcal{E}_r(z)=1$ with
 \be
 \mathcal{E}_r(z):=2\Big(2\mathrm{Bi}+r {T^*}^4\Big)zD_4(z)+\exp\left( 2\tfrac{\mu_M}{L_m}z\right) \dfrac{(2\mathrm{Bi}+r D_5)}{{\mu_M}}, \label{Epsilonr-radiactiva}
 \ee
where $D_4$ is given by \eqref{D1-4} and $D_5=4(T^*-T_m){|T^*|}^3,$
 then there exists a unique solution $f\in X$ for the  integral equation \eqref{EcIntegral-Radiactiva}.

 \end{teo}

 \begin{proof}
Let us split the proof into two steps. In the first one, we will see that $\mathcal{H}^r$ is a self-map of $X$ while in the second step we will see that it is a contracting mapping.

Let us show that $\mathcal{H}^r(X)\subset X$. Consider $f\in X$, then have that $0<F(f)(0)<{T^*}^4$ and so, 
$$
0<G(f)(0)\leq 2\Bi+r {T^*}^4 
$$ 
From \eqref{cota-E} and assumption \eqref{Hipotesis:Contraccion-radiactiva},  for every $\xi\in [0,\lambda]$ we can check that
\begin{align*}
&0\leq G(f)(0)\Big( \Phi(f)(\lambda)-\Phi(f)(\xi)\Big)<\dfrac{\Big(2\Bi+r {T^*}^4\Big)}{L_m} \bint_{\xi}^\lambda E(f)(z)\dd z \\[0.25cm]
&\leq \dfrac{\Big(2\Bi+r {T^*}^4 \Big) }{L_m}\bint_\xi^\lambda \exp\left( 2z\tfrac{\mu_M}{L_m}-z^2\tfrac{N_m}{L_M}\right)\dd z\\[0.25cm]
&\leq \dfrac{\Big(2\Bi+r {T^*}^4 \Big) }{L_m}\ \frac{\sqrt{\pi } \exp\left(\frac{\mu_M ^2 L_M}{L_m^2 N_m}\right)\left(\text{erf}\left(\frac{\frac{ \lambda  N_m}{L_M}-\frac{ \mu_M }{L_m}}{\sqrt{\frac{N_m}{L_M}}}\right)+\text{erf}\left(\frac{\frac{ \mu_M }{L_m}-\frac{ \xi  N_m}{L_M}}{ \sqrt{\frac{N_m}{L_M}}}\right)\right)}{2 \sqrt{\frac{N_m}{L_M}}}\\[0.25cm]
&\leq \dfrac{\Big(2\Bi+r {T^*}^4 \Big) }{L_m}\ \frac{\sqrt{\pi } \exp\left(\frac{\mu_M ^2 L_M}{L_m^2 N_m}\right)}{ \sqrt{\frac{N_m}{L_M}}}<1
\end{align*}
Therefore, $0\leq \mathcal{H}^{r}(f)(\xi)\leq 1$. It is clear that $\mathcal{H}^{r}(f)$ belongs to $C^0[0,\lambda]$, hence we get that $\mathcal{H}^{r}(f)\in X$  for every $f\in X$.

Now let us proceed to show that $\mathcal{H}^{r}$ is a contracting mapping. Consider $f_1$ and $f_2$ in $X$. 
From the mean value theorem applied to the function $g(x)=\Big( (T_m-T^*)x+T^* \Big)^4$, for $x_1=f_1(0)$ and $x_2=f_2(0)$ we have that
$$|g(x_1)-g(x_2)|=|g'(x^*)| |x_1-x_2|= 4\Big((T_m-T^*)x^*+T^* \Big)^3 |T_m-T^*| |x_1-x_2| ,$$
where $x^*$ is between $x_1$ and $x_2$, i.e. $0\leq x^*\leq 1$. 
As a consequence it follows that
\begin{align*}
&|F(f_1)(0)-F(f_2)(0)|\leq |g(x_1)-g(x_2)|\leq 4\Big((T_m-T^*)x^*+T^* \Big)^3 |T_m-T^*| |f_1(0)-f_2(0)|\\[0.25cm]
&\leq  4 |T^*-T_m| \;{|T^*|}^3\; ||f_1-f_2||= D_5 ||f_1-f_2||.
\end{align*}
and so
$$|G(f_1)(0)-G(f_2)(0)|\leq \Big( 2\; \Bi+r D_5\Big)\; ||f_1-f_2||.$$

Then, for each $0\leq \xi\leq \lambda$  we have that
\begin{align*}
&|\mathcal{H}^r(f_1)(\xi)-\mathcal{H}^{r}(f_2)(\xi)|\leq  |G(f_1)(0)| \;|\Phi(f_1)(\lambda)-\Phi(f_2)(\lambda)|+ |\Phi(f_2)(\lambda)|\; |G(f_1)(0)-G(f_2)(0)|\\[0.25cm]
&+|G(f_1)(0)| \;|\Phi(f_1)(\xi)-\Phi(f_2)(\xi)|+|\Phi(f_2)(\xi)|\; |G(f_1)(0)-G(f_2)(0)|\\[0.25cm]
&\leq \Big(2 (2\Bi+r {T^*}^4) \lambda D_4(\lambda)+\tfrac{(2\Bi+r D_5)}{\mu_M}  \exp\left(2\tfrac{\mu_M}{L_m}\lambda \right)\Big) ||f_1-f_2||=\mathcal{E}_{r}(\lambda)||f_1-f_2||.
\end{align*}
Under the assumption \eqref{Hipotesis:ExtraContraccion-radiactiva} we have that $\mathcal{E}_r$ satisfies the following properties
$$0<\mathcal{E}_r(0)= \dfrac{2\Bi+r D_5}{\mu_M}<1,\qquad\qquad \mathcal{E}_r(+\infty)=+\infty,\qquad\qquad \mathcal{E}_r'(z)>0,\qquad z\geq 0.$$
Therefore there exists a unique $\overline{\lambda}_r$ such that $\mathcal{E}_r(\overline{\lambda}_r)=1$. In addition, we obtain
$$\mathcal{E}_r(z)<1,\quad \forall \;0<z<\overline{\lambda}_r\qquad \text{and}\qquad \mathcal{E}_r(z)>1,\quad \forall z>\overline{\lambda}_r.$$
From the fixed Banach theorem we can state that for a fixed $\lambda\in (0,\overline{\lambda}_r)$ there exists a unique solution  $f\in X$ to the integral equation \eqref{EcIntegral-Radiactiva}.
 \end{proof}


For each given constant $0<\lambda<\overline{\lambda}_r$, the unique solution to equation \eqref{EcIntegral-Radiactiva}, $f(\xi)=f_\lambda(\xi)$ satisfies
\be
f_{\lambda}'(\xi)= G(f)(0)\dfrac{E(f_\lambda)(\xi)}{L^*(f_\lambda)(\xi)} . 
\ee
Then the condition \eqref{CondfLambda-Radiactiva} becomes equivalent to solve
\be
\mathcal{V}^r(\lambda)=\lambda,\label{PtoFijoLambda-Radiactivo}
\ee
where
\be 
\mathcal{V}^r(\lambda)=\mathcal{V}^r({f_\lambda},\lambda):= \dfrac{\Ste\; G(f)(0)}{2}  E(f)(\lambda)  .\label{Def-Vr-Radiactivo}
\ee

We can now state the following results. The proofs are omitted  due to the fact that they are obtained analogously to the results presented in the previous sections.
\medskip

\begin{lema} \label{LemaAcotaVr-Radiactivo}Assume that \eqref{Hipotesis-L*}-\eqref{Hipotesis-mu*} and \eqref{Hipotesis:Contraccion-radiactiva}-\eqref{Hipotesis:ExtraContraccion-radiactiva} hold.
Then for all $\lambda\in (0,\overline{\lambda}_r)$ we have that
\be
0<\mathcal{V}^r(\lambda)<\mathcal{V}_2^r(\lambda) \label{V1hV2h-Convectivo}
\ee
where $\mathcal{V}^r_2$ is given by
\begin{eqnarray}
\begin{array}{lll}
&\mathcal{V}_2^r(\lambda)=\dfrac{\mathrm{Ste} \; (2\mathrm{Bi}+r {T^*}^4)}{2} \exp\left(2\lambda \tfrac{\mu_M}{L_m}-\lambda^2 \tfrac{N_m}{L_M} \right),\qquad &\lambda>0.
\end{array}
\end{eqnarray}
Moreover, there exists at least one solution $\lambda_{2r}$ 
to the equation  
\be
\mathcal{V}^r_2(\lambda)=\lambda,\quad \lambda>0.\label{Lambda2-Radiactivo}
\ee
\end{lema}

\begin{teo} \label{TeoCota-Radiactivo} Assume that \eqref{Hipotesis-L*}-\eqref{Hipotesis-mu*} and \eqref{Hipotesis:Contraccion-radiactiva}-\eqref{Hipotesis:ExtraContraccion-radiactiva} hold.  Consider $\lambda_{2r}$ given by \eqref{Lambda2-Radiactivo}.
If $\mathcal{E}_r(\lambda_{2r})<1$, where $\mathcal{E}_r$ is defined by \eqref{Epsilonr-radiactiva},
  then there exists at least one solution $\widetilde{\lambda}_r\in (0,\lambda_{2r})$ to the equation \eqref{PtoFijoLambda-Radiactivo}.
\end{teo}

We return the original problem \eqref{Ec:Calor}, \eqref{Cond:RadiactivaConvectiva x=0}, \eqref{Cond:Stefan}-\eqref{Cond:FronteraInicial}. Notice that conditions \eqref{Hipotesis:Contraccion-radiactiva} and \eqref{Hipotesis:ExtraContraccion-radiactiva} can be rewritten as
 \be\label{Hipotesis-Extra-krhoc-Radiactivo}
 \dfrac{\left(2\Bi+r({T^*}^4-T_m^4)\right) \sqrt{k_0\rho_0 c_0 k_M}}{ k_m\sqrt{\gamma_m}} \sqrt{\pi} \exp\left(\tfrac{\nu_M^2 k_M}{ k_m^2 \gamma_m}\right)<1,\qquad \dfrac{\left(2\Bi+r D_5\right)\sqrt{\rho_0 c_0 k_0}}{\nu_M }<1.
 \ee

\begin{teo} Assume that  \eqref{Hipotesis-k}, \eqref{Hipotesis-rhoc}, \eqref{Hipotesis-v}  and \eqref{Hipotesis-Extra-krhoc-Radiactivo} hold. If $\mathcal{E}_r(\lambda_{2r})<1$ where $\mathcal{E}_r$ is given by \eqref{Epsilonr-radiactiva} and $\lambda_{2r}$ is given  by \eqref{Lambda2-Radiactivo}, then there exists at least one solution to the Stefan problem \eqref{Ec:Calor}, \eqref{Cond:RadiactivaConvectiva x=0}, \eqref{Cond:Stefan}-\eqref{Cond:FronteraInicial}, where the free boundary is given by
\be
s(t)=2\widetilde{\lambda}_r\sqrt{\alpha_0 t} ,\qquad t>0,
\ee
with $\widetilde{\lambda}_r$ defined by Theorem \ref{TeoCota-Radiactivo}, and the temperature is given by
\be
T(x,t)=(T_m-T^*)f_{\widetilde{\lambda}_r}(\xi)+T^* ,\qquad \qquad 0\leq \xi\leq \widetilde{\lambda}_r 
\ee
being $\xi=\tfrac{x}{2\sqrt{\alpha_0 t}}$ the similarity variable and $f_{\widetilde{\lambda}_r}$ the unique solution of the integral equation \eqref{EcIntegral-Radiactiva} which was established in Theorem \ref{TeoExistenciaf-Radiactiva}.
\end{teo}

\section{Particular cases}

\subsection{Constant thermal coefficients}

In this section we are going to recover the particular case analysed in \cite{Tu2018} that arises when we consider constant thermal coefficients, 
\be 
\rho(T)=\rho_0, \qquad  c(T)=c_0,\qquad  k(T)=k_0
\ee
and a velocity given by $v(T)=\frac{\mu(T)}{\sqrt{t}}$ with
\be \mu(T)=\rho_0 c_0 \sqrt{\alpha_0} \; \Pe
.\ee
 where $\Pe$ denotes the Peclet number.

Replacing those values in \eqref{LNv-Estrella} and \eqref{Def:muEstrella} we get that $L^*=N^*=1$ and $\mu^*=\Pe$. Then  $\Phi$, $E$, $U$ and $I$ defined by \eqref{Def-Phi}, \eqref{Def-E} and \eqref{Def-UI}, respectively become
\be 
\begin{array}{lll}
&U(f)(z)=\exp(2z \Pe),\qquad \qquad \qquad &I(f)(z)=\exp(z^2),\\ \\
& E(f)(z)=\exp(2 z\Pe-z^2)    & \Phi(f)(\xi)=\frac{\sqrt{\pi}\exp\left({\Pe^2}\right)}{2}   \Big( \erf(\Pe)-\erf(\Pe-\xi)\Big).
\end{array}
\ee
As a consequence we get that the explicit solution to the problem with a Dirichlet condition at the fixed face governed by \eqref{Ec:Calor}-\eqref{Cond:FronteraInicial},  is obtain trough the solution to the ordinary  differential problem \eqref{EcDif-Temp}-\eqref{EcLambda-Temp}, given by
\be
f(\xi)=\frac{\erf(\Pe)-\erf(\Pe-\xi)}{\erf(\Pe)-\erf(\Pe-\lambda)},\qquad 0\leq \xi\leq \lambda,
\ee
where $\lambda=\lambda(\Pe)$ is the unique solution to the following equation
\be
\Ste = \sqrt{\pi}\; \;\lambda \Big( \erf(\Pe)-\erf(\Pe-\lambda)\Big) \exp\left( (\Pe-\lambda)^2 \right).
\ee

\medskip
In a similar way we get that the problem with a Neumann condition governed by  \eqref{Ec:Calor},\eqref{Cond:Flujo x=0},\eqref{Cond:TempCambioFase}-\eqref{Cond:FronteraInicial} is equivalent to the ordinary differential problem \eqref{EcDif-f-flujo}-\eqref{EcLambda-flujo} whose  explicit solution (see \cite{Tu2018}) turns out to be
\be
f(\xi)=q^* \frac{\sqrt{\pi}\exp\left({\Pe^2}\right)}{2}   \Big( \erf(\Pe-\xi)-\erf(\Pe-\lambda)\Big),\qquad 0\leq \xi\leq \lambda,
\ee
where $\lambda$ is a solution to the following equation 
\be\label{Pec}
 \frac{q}{ \rho_0\ell\sqrt{\alpha_0}}=\lambda \exp\left(\lambda^2-2\lambda \Pe \right).
\ee

Notice that for $\Pe\leq \sqrt{2}$ it is obtained not only existence but also  uniqueness of solution to equation \eqref{Pec}.
\subsection{Linear thermal coefficients}


{\color{black} In this subsection we analyse the case where the thermal coefficients are given by
\be 
\rho(T)=\rho_0, \qquad  c(T)=c_0\Big( 1+\alpha \frac{T-T^*}{T_m-T^*}\Big),\qquad  k(T)=k_0\Big(1+\beta \frac{T-T^*}{T_m-T^*} \Big)
\ee
with $\alpha$ and $\beta$ given positive constants. In addition, we consider a velocity given by $v(T)=\frac{\mu(T)}{\sqrt{t}}$ with
\be \mu(T)=\rho_0 c(T) \sqrt{\alpha_0} \; \Pe
,\ee
 where $\Pe$ denotes the Peclet number.
 
This particular case appears in \cite{SiKuRa2019A}, where it is  considered the problem with a convective condition \eqref{Ec:Calor},\eqref{Cond:Convectiva x=0},\eqref{Cond:TempCambioFase}-\eqref{Cond:FronteraInicial}.

From \eqref{LNv-Estrella} and \eqref{Def:muEstrella} we get that $$L^*(f)=1+\beta f,\qquad \qquad N^*(f)=1+\alpha f,\qquad\qquad \mu^*(f)=\Pe \left(1+\alpha f \right).$$ 
Notice that as $f\in C^0[0,\lambda]$ with $0\leq f\leq 1 $, we get that  $L^*$, $N^*$ and $\mu^*$ verify
$$1\leq L^*(f)\leq 1+\beta,\qquad 1\leq N^*(f)\leq 1+\alpha, \qquad \Pe\leq \mu^*(f)\leq \Pe(1+\alpha)$$
In addition $L^*$, $N^*$ and $\mu^*$ satisfy hypothesis \eqref{Hipotesis-L*}-\eqref{Hipotesis-mu*} with $L_m=1$, $L_M=1+\beta$, $\widetilde{L}=\beta$; $N_m=1$, $N_M=1+\alpha$, $\widetilde{N}=\alpha$, $\mu_m=\Pe$, $\mu_M=\Pe(1+\alpha)$ and $\widetilde{\mu}=\Pe\, \alpha$.
\medskip

Moreover the  function  $\Phi$ defined by \eqref{Def-Phi} becomes

\be
\Phi(f)(\xi)=\bint_0^{\xi} \dfrac{\exp\left(2\bint_0^z  (\Pe-z) \tfrac{1+\alpha f(z)}{1+\beta f(z)}\right)}{1+\beta f(z)}\dd z.
\ee 
 Taking into account the hypothesis assumed in Theorem $\ref{teotemp}$, the coefficients $\alpha$, $\beta$ and the numbers $\Ste$, $\Pe$ must satisfy  the following condition
 \be
 2(1+\beta)\beta<1,
 \ee
that is, $0<\beta<\frac{\sqrt{3}-1}{2}$ and $\mathcal{E}(\lambda_2)<1$ where $\mathcal{E}$ is given by 
 \be
\mathcal{E}(\lambda)=2(1+\beta)\exp((1+\alpha)\lambda^{2})D_{4}(\lambda)
 \ee with
 $
 D_{4}(\lambda)
 $  given by \eqref{D1-4} and $\lambda_2$  defined by \eqref{Lambda2-temp}.

Those hypothesis are sufficient conditions in order to guarantee the existence of solution when a Dirichlet or Robin condition are imposed at the fixed face.

In case we consider $\Pe=0$  and  $\alpha=0$ we recover the problem studied in \cite{CeSaTa2018-a} and  \cite{BrNa2019} governed by
\begin{align}
&\Big((1+\beta f)f'(\xi)\Big)'+2 f'(\xi) \xi=0 ,& 0<\xi<\lambda \\
& (1+\beta f(0))f'(0)=2\Bi f(0),  \\
&  f(\lambda)=1, \\
& f'(\lambda)=\dfrac{2\lambda}{(1+\beta)\Ste},
\end{align}
where the existence of solution is obtained trough the Generalized Modified Error Function.

 }

\section{Conclusions}

We have  studied four different one-phase Stefan problems for a semi-infinite domain,  with the special feature of involving a moving phase change material as well as  temperature
dependent thermal coefficients.
All the problems that we have  analysed  were governed by the diffusion-convection  equation, where the uniform speed that appears in the convective term not only depends  on the temperature but also on time.
We have proved existence of at least one similarity solution  imposing  Dirichlet, Neumann, Robin or radiative-convective boundary condition at the fixed face. 
In each case, we have obtained an equivalent ordinary differential problem from where it was formulated an  integral equation coupled with a condition for the parameter that characterizes the free boundary. The system  obtained was solved though a double-fixed point analysis.
Moreover, we have provided the solutions to some particular problems that arise when we set the thermal coefficients to be constant or linear functions of the temperature.

\section*{Acknowledgement}
The present work has been partially sponsored by the Project PIP No 0275 from CONICET-UA, Rosario, Argentina, ANPCyT PICTO Austral 2016 No 0090 and 
the European Union's Horizon 2020 Research and
Innovation Programme under the Marie Sklodowska-Curie grant agreement 823731 CONMECH

\bibliographystyle{unsrt}
\bibliography{Biblio}

\end{document}